\documentclass{amsart}
\usepackage{amsmath,amsthm,latexsym,amssymb,hyperref,amsfonts}
\usepackage{fullpage}
\usepackage{stmaryrd}
\usepackage{eucal}
\usepackage{mathrsfs}
\usepackage[all]{xy}
\usepackage{tikz}
\usepackage{url}
\usepackage{tikz-cd}
\usepackage{enumerate}
\usepackage{enumitem}
\usetikzlibrary{matrix,arrows,decorations.pathmorphing}

\allowdisplaybreaks
\numberwithin{equation}{section}

\theoremstyle{definition}
\newtheorem{defi}{Definition}[section]
\newtheorem{property}[defi]{Property}
\newtheorem{ex}[defi]{Example}
\newtheorem{rem}[defi]{Remark}

\theoremstyle{plain}
\newtheorem{thm}[defi]{Theorem}
\newtheorem{prop}[defi]{Proposition}
\newtheorem{lem}[defi]{Lemma}

\newcommand{\C}{\mathcal C}
\newcommand{\D}{\mathcal D}
\newcommand{\I}{\mathbb I}
\newcommand{\J}{\mathbb J}
\newcommand{\T}{\mathsf{Top}_*}
\newcommand{\Smash}{\wedge}
\newcommand{\sm}{\wedge}
\newcommand{\id}{\mathsf{id}}
\renewcommand{\S}{\mathbb{S}}
\newcommand{\bS}{\mathbb{S}}

\newcommand{\bC}{\mathbb{C}}
\newcommand{\SpS}{\mathsf{Sp}^{\mathsf{\Sigma}}}

\newcommand{\Set}{\mathsf{Set}}
\newcommand{\sset}{\mathsf{sSet}}
\newcommand{\Top}{\mathsf{Top}}
\newcommand{\Mod}{\mathsf{Mod}}
\newcommand{\comon}{\mathsf{coMon}}

\newcommand{\Ev}{\mathsf{Ev}}
\newcommand{\ekmm}{\mathsf{Sp}^{\mathsf{EKMM}}}

\newcommand{\ssigma}{\mathsf{\Sigma}}

\newcommand{\W}{\mathcal{W}}
\newcommand{\SP}{\mathsf{SP}}
\newcommand{\DT}{\mathsf{Sp}^{\mathcal{D}}}

\long\def\emptytext#1{}

\begin{document}

\title [Coalgebras in symmetric monoidal categories of spectra]{Coalgebras in symmetric monoidal categories of spectra}
\author {Maximilien P\'eroux}
\address{Department of Mathematics, Statistics and Computer Science, University of Illinois at
Chicago,
851 S. Morgan Street,
Chicago, IL, 60607-7045, USA}
    \email{mholmb2@uic.edu}
    
 \author {Brooke Shipley}
    \email{shipleyb@uic.edu}

   

\begin{abstract}  
We show that all coalgebras over the sphere spectrum are cocommutative in the category of symmetric spectra, orthogonal spectra, $\Gamma$-spaces, $\mathcal{W}$-spaces and EKMM $\S$-modules. 
{Our result only applies to these strict monoidal categories of spectra and does not apply to the $\infty$-category setting.}
\end{abstract}

\maketitle 

\section{Introduction}

It is well known that the diagonal map of a set, or a space, gives it the structure of a comonoid. In fact,
the only possible (counital) comonoidal structure on an object in a Cartesian symmetric monoidal category is given by the diagonal (see \cite[Example 1.19]{AM}). Thus all comonoids are forced to be cocommutative in these settings. 
We prove that this rigidity is inherited by all of the strict monoidal categories of spectra that have been developed over the last 20 years, including symmetric spectra (see \cite{SS}), orthogonal spectra (see~\cite{MMSS, MM}), $\Gamma$-spaces (see \cite{Segal, BousfieldFrie}), $\mathcal{W}$-spaces (see \cite{Anderson}) and $\S$-modules (see \cite{EKMM}), which we call EKMM-spectra here.  That is, $\S$-coalgebra spectra in any of these categories are cocommutative, where $\S$ denotes the sphere spectrum.

\begin{thm}\label{thm intro: diag spectra coco}
Let $(C, \Delta, \varepsilon)$ be an $\bS$-coalgebra in symmetric spectra, orthogonal spectra, $\Gamma$-spaces, $\mathcal{W}$-spaces or {\upshape EKMM}-spectra.
Then $C$ is a cocommutative $\bS$-coalgebra. 
\end{thm}

Furthermore, we prove that all $R$-coalgebras are cocommutative whenever $R$ is a commutative $\S$-algebra with $R_0$ homeomorphic to $S^0$; see Theorems \ref{thm: diag spectra coco} and \ref{thm: ekmm coco}. 
Moreover, in Remark~\ref{rem-cof}, we recall that, in symmetric spectra or orthogonal spectra, every commutative $\S$-algebra is weakly equivalent to one whose zeroth space is homeomorphic to $S^0$.

It is important to note that these results are restricted to the listed strict monoidal categories of spectra.  In~\cite{lurie.elliptic}, Lurie considers the category of cocommutative coalgebras in a symmetric monoidal $\infty$-category and establishes an equivalence between the associated dualizable algebra and coalgebra objects in~\cite[Corollary 3.2.5]{lurie.elliptic}.  One could similarly consider the category of (not necessarily cocommutative) coalgebras and their dualizable objects. Since there exist compact {non-commutative} $\S$-algebra spectra, one can consider for such an $X$ the Spanier-Whitehead dual, $DX = \hom(X, \bS)$, and this represents a {non-cocommutative} $\S$-coalgebra spectrum. This shows that the category of coalgebras in the $\infty$-category of spectra has non-cocommutative objects.

Given a symmetric monoidal model category $\mathcal{M}$, the underlying $\infty$-category, $\overline{\mathcal{M}},$ admits a symmetric monoidal structure.  As part of his PhD thesis, the first author is determining conditions that imply the equivalence of the $\infty$-category of (cocommutative) coalgebras over $\overline{\mathcal{M}}$ and the nerve of (cocommutative) coalgebras in $\mathcal{M}$. 
In other words, this compares the coalgebra objects {with a (cocommutative) 
comultiplication up to coherent homotopy} with the strictly (cocommutative) 
coalgebras.

The paper is organized as follows. 
In Section \ref{sec: definitions}, we recall the notions of comonoids and coalgebras and introduce a purely categorical argument, Theorem \ref{thm: prethm}, which is at the heart of the proof of Theorem \ref{thm intro: diag spectra coco}. 
Section \ref{sec: diagram spectra} considers the general setting of diagram spectra introduced in \cite{MMSS} which encapsulates most of the categories of spectra mentioned above. 
The category of EKMM-spectra needs particular care and is considered in Section \ref{sec: ekmm}.

\subsection*{Acknowledgements}  We are grateful for initial conversations on these topics with Kathryn Hess. We also thank Ben Antieau for reading and commenting on early versions of this paper. The second author was supported by the National Science Foundation, DMS-140648. { The authors would like to thank the referee for helpful comments on the preliminary version of this paper.}

\section{Definition and Preliminaries}\label{sec: definitions}

Let $(\C, \otimes, \mathbb{I})$ be a symmetric monoidal category throughout this section.

\begin{defi}
A \emph{comonoid} $(C, \Delta, \varepsilon)$ in $\C$ consists of an object $C$ in $\C$ together with a coassociative comultiplication $\Delta : C\rightarrow C\otimes C$, such that the following diagram commutes: 
\[
\begin{tikzcd}[column sep =large]
C\ar{r}{\Delta} \ar{d}[swap]{\Delta} &C\otimes C\ar{d}{\mathsf{id}_C\otimes \Delta}\\
C\otimes C \ar{r}{\Delta\otimes \mathsf{id}_C} & C\otimes C\otimes C,
\end{tikzcd}
\]
and admits a counit morphism $\varepsilon : C\rightarrow \mathbb{I}$ such that we have the following commutative diagram: 
\[
\begin{tikzcd}[row sep=small]
C\otimes C \ar{r}{\mathsf{id}_C\otimes \varepsilon} & C\otimes \mathbb{I} \cong C \cong \mathbb{I}\otimes C & C\otimes C \ar{l}[swap]{\varepsilon \otimes \mathsf{id}_C}\\
& C.\ar[equals]{u}\ar[bend left]{ul}{\Delta} \ar[bend right]{ur}[swap]{\Delta} &
\end{tikzcd}
\]
The comonoid is \emph{cocommutative} if the following diagram commutes: 
\[
\begin{tikzcd}
C\otimes C \ar{rr}{\tau} && C\otimes C\\
& C,\ar{ul}{\Delta}\ar{ur}[swap]{\Delta} &
\end{tikzcd}
\]
where $\tau$ is the twist isomorphism from the symmetric monoidal structure of $\C$. A morphism of comonoids $f:(C,\Delta, \varepsilon)\rightarrow (C',\Delta',\varepsilon')$ is a morphism $f:C\rightarrow C'$ in $\C$ such that the following diagrams commute: 
\[
\begin{tikzcd}
C\ar{r}{f} \ar{d}[swap]{\Delta} & C'\ar{d}{\Delta'} & C\ar{r}{f}\ar{dr}[swap]{\varepsilon} & C'\ar{d}{\varepsilon'}\\
C\otimes C \ar{r}{f\otimes f} & C'\otimes C', & & \I.
\end{tikzcd}
\]
We denote $\comon(\C)$ the category of comonoids in $\C$.
\end{defi}

In the next sections, our main strategy uses the following result. 

\begin{thm}\label{thm: prethm}
Let $(\C, \otimes, \mathbb{I})$ and $(\D, \odot, \mathbb{J})$ be symmetric monoidal categories endowed with a pair of adjoint underlying functors 
$\begin{tikzcd}
L:\C \ar[shift left]{r} & \ar[shift left]{l} \D:R,
\end{tikzcd}$ where the functor $R$ is lax symmetric monoidal such that: 
\begin{enumerate}[label=\upshape\textbf{(\roman*)}]
\item \label{cond1 of thm} the maps $\I\rightarrow R(\J)$ and $R(D)\otimes R(D)\rightarrow R(D\odot D)$ are isomorphisms in $\D$, for all comonoids $D$ in $\D$;
\item \label{cond2 of thm} all comonoids in $\C$ are cocommutative;
\item \label{cond3 of thm} for each comonoid $D$ in $\D$, the counit map $LR(D)\rightarrow D$ given by the adjunction is an epimorphism in $\D$.
\end{enumerate}
Then all comonoids in $\D$ are cocommutative.
\end{thm}

\noindent
Notice that the condition \ref{cond1 of thm} of Theorem \ref{thm: prethm} is respected whenever $R$ is a strong monoidal functor.

\begin{proof}
Recall from \cite[Proposition 3.85]{AM}, as $R:\D\rightarrow \C$ is a lax symmetric monoidal functor, there exists a unique lax symmetric comonoidal structure on the functor $L:\C\rightarrow \D$ such that the adjoint pair forms a symmetric monoidal conjunction (sometimes called colax-lax adjunction, see \cite[Definition 3.81]{AM}). Since $L$ is lax symmetric comonoidal, it sends (cocommutative) comonoids in $\C$ to (respectively cocommutative) comonoids in $\D$ (see \cite[Proposition 3.29, Proposition 3.37]{AM}). 
Because of condition \ref{cond1 of thm}, one can check that if $(D,\Delta, \varepsilon)$ is a comonoid in $\D$, then $R(D)$ is a comonoid in $\C$ with comultiplication:
\[
\begin{tikzcd}
R(D)\ar{r}{R(\Delta)} & R(D\odot D) & R(D)\otimes R(D) \ar{l}[swap]{\cong},
\end{tikzcd}
\]
and counit:
\[
\begin{tikzcd}
R(D) \ar{r}{R(\varepsilon)} & R(\J) & \I\ar{l}[swap]{\cong}.
\end{tikzcd}
\]
Using condition \ref{cond2 of thm} we get that $LR(D)$ is a cocommutative comonoid in $\D$ for any (not necessarily cocommutative) comonoid $D$ in $\D$. 
One can show the counit map of the adjoint $LR(D)\rightarrow D$ is a morphism of comonoids in $\D$ as in \cite[Proposition 3.93]{AM}. 
We conclude using \ref{cond3 of thm} and Proposition \ref{prop: cocommutative + epi}.
\end{proof}

Recall that any subalgebra of a commutative algebra is also commutative. The following is the dual case and could be proved using opposite categories.

\begin{prop}\label{prop: cocommutative + epi}
Let $(C,\Delta, \varepsilon)$ and $(C',\Delta', \varepsilon')$ be comonoids in $(\C,\otimes, \I)$. Suppose $C$ is cocommutative. 
Given a morphism of comonoids $f:C\rightarrow C'$, if $f$ is an epimorphism in $\C$, then $C'$ is also cocommutative.
\end{prop}

\begin{proof}
Since $f$ is a morphism of comonoids the top square in the following diagram commutes:
\[
\begin{tikzcd}
C\ar{r}{f} \ar{d}{\Delta}\ar[bend right=60]{dd}[swap]{\Delta} & C' \ar{d}{\Delta'} \\
C\otimes C \ar{r}{f\otimes f} \ar{d}{\tau} & C'\otimes C'\ar{d}{\tau}\\
C\otimes C \ar{r}{f\otimes f} & C'\otimes C'.
\end{tikzcd}
\]
The bottom square commutes from the naturality of the twist isomorphism $\tau$. The left side commutes as $C$ is cocommutative.  The commutativity of the above diagram gives:
\[
\tau\circ \Delta ' \circ f= \tau \circ (f\otimes f) \circ \Delta = (f\otimes f)\circ \tau \circ \Delta= (f\otimes f)\circ \Delta= \Delta'\circ f.
\]
Since $f$ is an epimorphism in $\C$, it follows that $\Delta'=\tau \circ \Delta'$.
Therefore $(C',\Delta', \varepsilon')$ is cocommutative.
\end{proof}

Subsequently the condition \ref{cond2 of thm} of Theorem \ref{thm: prethm} will be verified using the following lemma.
Let $(\Top, \times, *)$ be the category of spaces (weak Hausdorff $k$-spaces) endowed with the Cartesian product and $(\Top_*, \Smash, S^0)$ the based spaces endowed with the smash product, where $S^0=\{0,1\}$ is the unit, with $0$ as basepoint. Recall that the functor:
\begin{eqnarray*}
(-)_+:(\Top, \times, *) & \longrightarrow & (\Top_*,\Smash, S^0)\\
X & \longmapsto & X \coprod \{ *\}
\end{eqnarray*}
is strong symmetric monoidal: $X_+\Smash Y_+\cong (X\times Y)_+$. Since $(\Top, \times, *)$ is a Cartesian symmetric monoidal category, every space $C$ has a unique comonoidal structure with respect to the Cartesian product, see \cite[Example 1.19]{AM}. The counit $\varepsilon: C\rightarrow *$ is the unique map to the terminal object, and the comultiplication $\Delta:C\rightarrow C\times C$ is the diagonal $\Delta=(\id_C, \id_C)$. The comultiplication is always cocommutative. Since the functor $(-)_+:\Top\rightarrow \Top_*$ is strong symmetric monoidal, it sends the cocommutative comonoid $(C,\Delta, \varepsilon)$ to a cocommutative comonoid $(C_+,\Delta_+, \varepsilon_+)$. The next result says that these are the only possible comonoids in $(\Top_*, \Smash, S^0)$. It turns out that this is purely a point-set argument, and the result remains valid for the category of sets, and simplicial sets, denoted respectively $\Set$ and $\sset$.

\begin{lem}\label{lem: coalg in set}
Let the category $\C$ be either $\Set$, $\sset$ or $\Top$ endowed with its Cartesian symmetric monoidal structure. Then the faithful strong symmetric monoidal functor $(-)_+:(\C,\times, *)\rightarrow (\C_*, \Smash, S^0)$ lifts to an equivalence of categories $(-)_+:\comon(\C)\rightarrow \comon(\C_*)$. In particular, any comonoid in $\C_*$ is cocommutative and isomorphic to a certain comonoid $(C_+, \Delta_+, \varepsilon_+)$ where $(C,\Delta,\varepsilon)$ is a cocommutative comonoid in $\C$ with the diagonal as a comultiplication.
\end{lem}

\begin{proof}
We argue only for $\C=\Set$ and claim the other cases are similar.
The functor is clearly faithful. Let us first show that it is essentially surjective on comonoid objects.
Let $(C', \Delta', \varepsilon')$ be a comonoid in $\Set_*$. We first argue that $\Delta'(c)=c\Smash c$ for all $c\neq *$ in $C$.
Let us denote $\Delta'(c)=c_1 \Smash c_2$. 
 
If $c_1=*$, then $(\mathsf{id}_C'\Smash \varepsilon')(c_1\Smash c_2)=*$, as the map $\mathsf{id}_C'\Smash \varepsilon'$ is pointed, and thus counitality of $C'$ shows:
\begin{equation}\label{equ: comonoid in set}
c=((\mathsf{id}_C'\Smash \varepsilon')\circ \Delta') (c)=(\mathsf{id}_C'\Smash \varepsilon')(c_1\Smash c_2)=*,
\end{equation}
which is a contradiction with $c\neq *$. Thus $c_1\neq *$, and similarly one can show $c_2\neq *$ when $c\neq *$. 

Let us show $\varepsilon'(c_1)= 1$ and $\varepsilon'(c_2)= 1$, when $c\neq *$. If we assume $\varepsilon'(c_2)= 0$, then we obtain again equation (\ref{equ: comonoid in set}) which is a contradiction with $c\neq *$. Thus $\varepsilon'(c_2)\neq 0$ and we prove similarly $\varepsilon'(c_1)\neq 0$ when $c\neq *$.

Let us prove $\Delta(c)=c\Smash c$, for $c\neq *$.
Since $C'$ is counital, we get: 
\[
c = ((\mathsf{id}_C'\Smash \varepsilon')\circ \Delta') (c)=c_1\Smash \varepsilon'(c_2)=c_1\Smash 1.
\] So $c_1=c$. Similarly, $c_2=c$.

Now notice that $C'\cong \varepsilon'^{-1}(0)\coprod \varepsilon'^{-1}(1)$. Let $c\in \varepsilon'^{-1}(0)$. Then from the equation:
\[
((\mathsf{id}_{C'}\Smash \varepsilon')\circ \Delta') (c)=(\mathsf{id}_{C'}\Smash \varepsilon')(c\Smash c)=c\Smash 0=*,\]
counitality of $C'$ concludes that $c=*$. Thus $\varepsilon'^{-1}(0)=*$. 
Regard $C=\varepsilon'^{-1}(1)$ as an object in $\Set$.
Denote the diagonal on $C$ by $\Delta=(\id_C,\id_C):C\rightarrow C\times C$ and $\varepsilon: C\rightarrow *$ the unique map to the point. 
We have just shown that $(C_+, \Delta_+, \varepsilon_+)$ is isomorphic to the comonoid $(C', \Delta', \varepsilon')$ in $\Set_*$.
This proves that the functor $(-)_+$ is essentially surjective on comonoids. 

Let us show the functor $(-)_+:\Set\rightarrow \Set_*$ is full on comonoid objects. Given two comonoids denoted by $(C_+, (\Delta_C)_+, (\varepsilon_C)_+)$ and $(D_+, (\Delta_D)_+, (\varepsilon_D)_+)$ in $\Set_*$, let $f:C_+\rightarrow D_+$ be a map of comonoids in $\Set_*$. Then counitality gives the commutative diagram:
\[
\begin{tikzcd}[column sep=small]
C_+ \ar{rr}{f}\ar{dr}[swap]{(\varepsilon_C)_+} & & D_+\ar{dl}{(\varepsilon_D)_+}\\
& S^0. &
\end{tikzcd}
\]
If, for $c\neq *$ in $C_+$, we have $f(c)=*$, then commutativity of the diagram gives:
\[
1=(\varepsilon_C)_+(c)=((\varepsilon_D)_+\circ f)(c)=(\varepsilon_D)_+(*)=0,
\]
which is a contradiction. Thus the map $f:C_+\rightarrow D_+$ is induced by a map $C\rightarrow D$ in $\Set$, which proves that $(-)_+:\Set \rightarrow \Set_*$ is full on comonoid objects.
\end{proof}

\begin{rem}
Subsequently, we will only be using that all comonoids in the symmetric monoidal category $(\T, \Smash, S^0)$ are cocommutative. Notice that Theorem \ref{thm: prethm} does \emph{not} apply for the adjoint pair of functors:
\[
\begin{tikzcd}
(-)_+:(\Top, \times, *) \ar[shift left]{r} & (\T, \Smash, S^0):U \ar[shift left]{l},
\end{tikzcd}
\]
where $U:\T\rightarrow \Top$ is the forgetful functor, as $U$ does not respect condition \ref{cond1 of thm} of Theorem \ref{thm: prethm}.
\end{rem}

Recall that given a commutative monoid $R$ in $\C$, the category of (left) modules over $R$ in $\C$, denoted $\mathsf{Mod}_R(\C)$ is a symmetric monoidal category, where the unit is $R$ and the monoidal product is denoted $\otimes_R$ and is defined as the coequalizer:
\[ 
\begin{tikzcd}[row sep=large, column sep = huge]
 M\otimes R\otimes N \ar[shift left, "\id_M \otimes \alpha_N"]{r}
 \ar[shift right, "(\alpha_M\circ \tau) \otimes \id_N"']{r} & M \otimes N, 
 \end{tikzcd} 
\]
where $\alpha_M:R\otimes M \rightarrow M$ and $\alpha_N:R\otimes N\rightarrow N$ are the (left) $R$-actions on $M$ and $N$ respectively. 
This leads to the following definition.

\begin{defi}
Let $R$ be a commutative monoid in $\C$. A \emph{coalgebra $(C, \Delta, \varepsilon)$ over $R$ in $\C$}, or an \emph{$R$-coalgebra in $\C$}, is a comonoid $(C, \Delta, \varepsilon)$ in the symmetric monoidal category $( \mathsf{Mod}_R(\C), \otimes_R, R)$. A morphism $f:(C, \Delta, \varepsilon)\rightarrow (C',\Delta', \varepsilon')$ of $R$-coalgebras in $\C$ is a morphism of comonoids in $\mathsf{Mod}_R(\C)$. 
\end{defi}

\section{Coalgebras in Diagram Spectra}\label{sec: diagram spectra}

Let us recall the definitions from \cite{MMSS} and set the notation. 
Let $\D=(\D, \otimes, 0)$ be a locally small symmetric monoidal based topological category with unit object $0$ and continuous monoidal product $\otimes$, with base point $*$.
Let $\T$ be the category of based spaces (weak Hausdorff $k$-spaces). Recall that a \emph{$\D$-space} $X$ is a continuous based functor $X:\D\rightarrow \T$. 
If $X$ and $Y$ are $\D$-spaces, their (internal) smash product $X\Smash Y$ is a $\D$-space such that, for each object $d$ in $\D$, we have:
\begin{equation}\label{equ: smash in diagrams}
(X\wedge Y)(d) = \mathsf{colim}_{e\otimes f \rightarrow d} \Big(X(e)\wedge Y(f)\Big),
\end{equation}
where the colimit is taken over the commutative triangles:
\begin{equation}\label{diag: colimit over d}
\begin{tikzcd}
e'\otimes f' \ar{rr}{\varphi \otimes \psi}\ar{dr} & & e\otimes f\ar{dl}\\
& d. &
\end{tikzcd}
\end{equation}
As $\D$ is locally small, we can interpret the above colimit as a coend, i.e., $(X\Smash Y)(d)$ is the following coequalizer in $\T$:

\begin{equation}\label{equ: colimit coend}
\begin{tikzpicture}[baseline= (a).base]
\node[scale=.95] (a) at (0,0){
\begin{tikzcd}[column sep=huge]
\displaystyle  \! \! \! \! \bigvee_{\substack{\scriptstyle(\varphi, \psi)\, \mathsf{in} \\ \mathsf{Mor}(\D\times \D)}} \! \! \! \D(e\otimes f, d) \Smash X(e')\Smash Y(f')\ \  \ar[shift left]{r}{\scriptstyle(\varphi\otimes \psi)^*\Smash \id \Smash \id}\ar[shift right]{r}[swap]{\scriptstyle\id\Smash X(\varphi)\Smash Y(\psi)} & \displaystyle\bigvee_{\substack{\scriptstyle (e,f)\,\mathsf{in} \\ \mathsf{Ob}(\D\times \D)}}  \! \! \!  \D(e\otimes f, d) \Smash X(e) \Smash Y(f).
\end{tikzcd}
};
\end{tikzpicture}
\end{equation}
See more detail in \cite[Definition 21.4]{MMSS}. 
Then $\D$-spaces form a symmetric monoidal category denoted $\T^{\D}$ (see \cite[Theorem 1.7]{MMSS} where the category is denoted $\D\mathcal{T}$). 

A commutative monoid in $\T^{\D}$ is a lax symmetric monoidal functor $\D\rightarrow \T$. 
Let $R$ be a commutative monoid in $\T^{\D}$ with unit $\lambda:S^0\rightarrow R(0)$, and product $\phi: R(d)\Smash R(e)\rightarrow R(d\otimes e)$, for any $d$ and $e$ in $\D$.
A \emph{$\D$-spectrum $X$ over $R$} is an $R$-module in $\T^{\D}$. It is a $\D$-space $X:\D\rightarrow \T$ together with continuous maps $\sigma : R(d)\Smash X(e) \rightarrow X(d\otimes e)$, natural in $d$ and $e$, such that the composite:
\[ 
X(d)\cong S^0\Smash X(d) \stackrel{\lambda\Smash\id}\longrightarrow R(0)\Smash X(d) \longrightarrow X(0\otimes d)\cong X(d),
\] 
is the identity and the following diagram commutes:
\[
\begin{tikzcd}
 R(e) \Smash R(f)\Smash X(d)\ar{d}[swap]{\phi\Smash \id} \ar{r}{\id\Smash \sigma} & R(e)\Smash X(f\otimes d) \ar{d}{\sigma}\\
 R(e\otimes f)\Smash X(d) \ar{r}[swap]{\sigma} & X(e\otimes f\otimes d).
\end{tikzcd}
\]
As recalled in the previous section, the smash product $X\Smash_R Y$ of two $R$-modules $X$ and $Y$ is defined as the coequalizer:
\[
\begin{tikzcd}[column sep=huge]
 X\Smash R\Smash Y \ar[shift left]{r}{\id_X\Smash \sigma_Y}\ar[shift right]{r}[swap]{(\sigma_X\circ \tau) \Smash \id_Y} & X\Smash Y,\end{tikzcd}
\]
where $\tau$ is the twist isomorphism. We denote the category $\Mod_R(\T^{\D})$ of $R$-modules simply by $\DT_R$.
Subsequently we assume the following conditions on the topological category $\D$.

\begin{property}\label{property1}
There is a faithful strong symmetric monoidal continuous based functor $\S: \D\rightarrow \T$.
\end{property}

\begin{property}\label{property2} 
If $\S(d)$ is homeomorphic to the base point in $\T$, then $d$ is isomorphic to the base point $*$ in $\D$. If $\S(d)$ is homeomorphic to $S^0$, then $d$ is isomorphic to the unit object $0$ in $\D$. 
\end{property}  

\begin{ex}
Recall, from \cite[Examples 4.2, 4.4, 4.6, 4.8]{MMSS}, that particular choices of the category $\D$ recover the usual definition of the categories of symmetric spectra, orthogonal spectra, $\Gamma$-spaces and $\mathcal{W}$-spaces, among others.  It is also shown that there exist faithful strong symmetric monoidal based functors $\S: \D\rightarrow \T$ for each of the choices of $\D$, which of course identify with the usual sphere spectrum definition. It is elementary to check that Property \ref{property2} is verified in each of these cases. There may be other examples of interest.
\end{ex}

We now state the main theorem.

\begin{thm}\label{thm: diag spectra coco}
Let $\D$ satisfy Properties {\normalfont\ref{property1}} and {\normalfont\ref{property2}}.
Let $(R,\phi, \lambda)$ be a commutative monoid in $\T^{\D}$, where $R(0)\cong S^0$. Let $(C, \Delta, \varepsilon)$ be an $R$-coalgebra in $\T^{\D}$. Then $C$ is a cocommutative $R$-coalgebra. 
In particular, all coalgebras over the sphere spectrum in symmetric spectra, orthogonal spectra, $\Gamma$-spaces and $\mathcal{W}$-spaces are cocommutative.
\end{thm} 

\begin{rem}\label{rem-cof}
We show here that, in symmetric spectra or orthogonal spectra, any cofibration of commutative $\bS$-algebras is the identity at level 0. For a cofibrant object $R$, the unit map $\bS \to R$ is a cofibration and hence it follows that $R(0) \cong \bS(0) = S^0$. {Thus, in symmetric spectra or orthogonal spectra, any cofibrant commutative $\bS$-algebra $R$ has the property that $R(0) \cong S^0$. }

Recall the free commutative $\bS$-algebra functor, denoted $\bC(X) = \bigvee_{n\geq 0} X^{\sm n} / \Sigma_n$, with $X^{\sm 0} = \bS$. 
The cofibrations in the model structure on commutative $\bS$-algebras from~\cite[15.1]{MMSS} are built by applying $\bC$ to the positive cofibrations defined in~\cite[14.1]{MMSS}. 
There it is noted that these positive cofibrations are homeomorphisms at level 0. 
In fact, by~\cite[6.2]{MMSS}, any generating positive cofibration is of the form $F_k i$ with $i$ an $h$-cofibration and $k > 0$. 
Since $k > 0$, level 0 of these maps is the identity map on the trivial one point space~$*$.  
Since $X{\sm n}(0) \cong X(0){\sm n}$, the only contribution to level 0 for the map $\bC F_k i$ is the identity map on $S^0$ coming from the summand with $n=0$.  
Thus, any cofibration of commutative $\bS$-algebras is the identity on level 0. 

This also holds in the $\bS$-model structure from~\cite[3.2]{shipley-convenient} for commutative $\bS$-algebras. There a cofibration of commutative $\bS$-algebras is shown to be an underlying positive $\bS$-cofibration of $\bS$-modules by~\cite[4.1]{shipley-convenient}. 
These maps are isomorphisms in level 0 by definition~\cite[Section 3]{shipley-convenient}.  
Elsewhere these model structures are referred to as the "flat" or "positive flat" model structures~\cite{schwede-book}.
\end{rem}

We wish to prove Theorem \ref{thm: diag spectra coco}. 
We will use Theorem \ref{thm: prethm} for the following adjoint pair of functors. Given a commutative monoid $R$ in $\T^{\D}$, the free $R$-module functor is the left adjoint to the evaluation at the unit object $0$ of $\D$ (see \cite[Definition 1.3]{MMSS}):
\[
\begin{tikzcd}
R\Smash -: \T \ar[shift left]{r} & \DT_R:\Ev_0.\ar[shift left]{l}
\end{tikzcd}
\]
Equation (\ref{equ: colimit coend}) shows that there is a natural map $X(0)\Smash Y(0)\rightarrow (X\Smash_R Y)(0)$, for any $X$ and $Y$ in $\DT_R$, which makes the functor $\Ev_0:\DT_R\rightarrow \T$ lax symmetric monoidal. We have already shown in Lemma \ref{lem: coalg in set} that condition \ref{cond2 of thm} of Theorem \ref{thm: prethm} is verified by $(\T, \Smash, S^0)$. So we only need to know when conditions \ref{cond1 of thm} and \ref{cond3 of thm} of Theorem \ref{thm: prethm} are verified. 

Let us first investigate the condition \ref{cond3 of thm} of Theorem \ref{thm: prethm}.
Notice that a map $f:X\rightarrow Y$ in $\T^{\D}$ is an epimorphism if and only if, for any object $d$ in $\D$, the map $f:X(d)\rightarrow Y(d)$ is an epimorphism in $\T$, i.e. a surjective based continuous map. 
 
\begin{lem}\label{lem: epimorphism}
Let $\D$ satisfy Properties {\normalfont\ref{property1}} and {\normalfont\ref{property2}}.
Let $(R,\phi, \lambda)$ be a commutative monoid in $\T^{\D}$. Let $(C, \Delta, \varepsilon)$ be an $R$-coalgebra in $\T^{\D}$. 
Then the natural map $R\Smash C(0) \longrightarrow C$ is an epimorphism of $R$-modules in $\T^{\D}$.
\end{lem}

The main idea of the proof is to look at the consequences of counitality of an $R$-coalgebra $C$ with respect to the identifications in the smash product $C\Smash_R C$. 
Before proving the lemma, we need the following result.

\begin{lem}\label{lem: smash in simp set}
Let $X$ be an object of $(\T, \Smash, S^0)$.
\begin{enumerate}[label=\upshape\textbf{(\roman*)}]
\item\label{item1 lem} If we are given pointed maps $f:S^0\rightarrow X$ and $g:X\rightarrow S^0$ such that the induced map:
\[
\begin{tikzcd}
S^0\Smash X \ar{r}{f\Smash g}[swap]{\cong} & X\Smash S^0
\end{tikzcd}
\]
is an isomorphism, then either $X\cong *$ and the morphisms $f$ and $g$ are the trivial maps, or $X\cong S^0$ and $f$ and $g$ are isomorphisms.
\item\label{item2 lem} Suppose we are given pointed sets $Y$ and $Z$ together with pointed maps $f:S^0\rightarrow Y$, $g:X\rightarrow Z$, $f':X\rightarrow Y$ and $g':S^0\rightarrow Z$ such that the following diagram commutes: 
\[
\begin{tikzcd}
S^0\Smash X \ar{r}{f\Smash g}\ar{dr}[swap]{\cong} & Y\Smash Z \ar{d}{\alpha} & X\Smash S^0\ar{l}[swap]{f'\Smash g'}\ar{dl}{\cong}\\
& X &
\end{tikzcd}
\]
for a pointed map $\alpha: Y\Smash Z \rightarrow X$. Then either $X\cong *$ and the morphisms $g$ and $f'$ are trivial, or $X\cong S^0$ and the composite $\begin{tikzcd}[column sep= small]
S^0\cong S^0\Smash S^0 \ar{r}{f\Smash g'} & Y\Smash Z\ar{r}{\alpha} & X
\end{tikzcd}$ is an isomorphism.
\end{enumerate}
\end{lem}

\begin{proof}
The proof is purely a point set argument.
For \ref{item1 lem}, assume $X\ncong *$. Let $x\neq *$ in $X$. The commutativity of the diagram:
\[
\begin{tikzcd}
S^0\Smash X \ar{rr}{f\Smash g}[swap]{\cong}\ar{dr}[swap]{\mathsf{id}_{S^0}\Smash g} & & X\Smash S^0\\
& S^0\Smash S^0,\ar{ur}[swap]{f\Smash \mathsf{id}_{S^0}} &
\end{tikzcd}
\]
implies that $f\Smash \mathsf{id}_{S^0}$ is surjective. Thus any element $x\Smash 1$ in $X\Smash S^0$ is of the form $f(1)\Smash 1$. This implies $X\cong S^0$ .

For \ref{item2 lem}, assume again $X\ncong *$. The map $\alpha$ is surjective as $\alpha \circ (f\Smash g)$ is an isomorphism. For an element $x\neq *$ in $X$, denote by $y\Smash z$ an element in $Y\Smash Z$ such that $\alpha(y\Smash z)=x$. Then we get:
\[
f(1)\Smash g(x) = y\Smash z = f'(x) \Smash g'(1).
\]
Thus $y=f(1)=f'(x)$ and $z=g'(1)=g(x)$. The desired composite is an isomorphism. Whence $X\cong S^0$.
\end{proof}

\begin{proof}[Proof of Lemma \ref{lem: epimorphism}]
We need to prove that the continuous structure map $\sigma: R(d)\Smash C(0) \rightarrow C(d)$ is surjective for each $d$ in $\D$.
If $d=0$, since the composition $C(0)\cong S^0\Smash C(0) \longrightarrow R(0)\Smash C(0) \stackrel{\sigma}\longrightarrow C(0)$ is the identity, then $\sigma: R(0)\Smash C(0) \rightarrow C(0)$ is surjective. If $d=*$, then the map $\sigma:R(*)\Smash C(0)\longrightarrow C(*)$ is trivial, as the functors $C$ and $R$ are pointed. 

Let us assume now that $d$ is an object in $\D$, where $d\ncong 0,*$.
The above definition of the smash product of $R$-modules in $\T^{\D}$ leads to the explicit definition of $C\wedge_R C$ as a quotient space of:
\begin{equation}\label{equ C wedge C in diag}
(C\wedge_R C)(d) = \left.\left(\bigvee_{(e,f)\in \D\times \D} \D(e\otimes f, d)\Smash C(e) \Smash C(f)\right)\right/\sim_R,
\end{equation}
for any $d$ in $\D$. 
The relations here are induced by the internal smash product (see coequalizer \ref{equ: colimit coend} {for $X=Y=C$}) and by the $R$-action via the structure maps $\sigma : R(e)\Smash C(f) \rightarrow C(e\otimes f)$.  
Notice that the natural isomorphisms $d\otimes 0 \stackrel{\cong}\rightarrow d$ and $0\otimes d \stackrel{\cong}\rightarrow d$ imply that there is at least one copy of $C(d)\Smash C(0)$ and $C(0)\Smash C(d)$ in $(C\Smash_R C)(d)$. 

First let us show that an element in $C(d)\Smash C(0)$ is not identified with an element in $C(0)\Smash C(d)$ in ${(C\Smash C)}(d)$ via the coequalizer (\ref{equ: colimit coend}). 
For this matter, assume there are objects $e$ and $f$ in $\D$ that fit in either of the following commutative diagrams in $\D$:
\[
\begin{tikzcd}
d\otimes 0 \ar{r}\ar{dr}[swap]{\cong} & e\otimes f \ar{r}\ar{d} & 0\otimes d \ar{dl}{\cong} & d\otimes 0 \ar{r}\ar{dr}[swap]{\cong} & e\otimes f \ar{d} & 0\otimes d \ar{dl}{\cong}\ar{l}\\
& d, & & & d. &
\end{tikzcd}
\]
Using Property \ref{property1}, we apply the functor $\S:\D\rightarrow \T$ and get the commutative diagrams in $\T$:
\[
\begin{tikzcd}[column sep=10]
\S(d)\Smash S^0 \ar{r}\ar{dr}[swap]{\cong} & \S(e)\Smash \S(f) \ar{r}\ar{d} & S^0\Smash \S(d) \ar{dl}{\cong} & \S(d)\Smash S^0 \ar{r}\ar{dr}[swap]{\cong} & \S(e)\Smash \S(f) \ar{d} & S^0\Smash \S(d) \ar{dl}{\cong}\ar{l}\\
& \S(d), & & & \S(d). &
\end{tikzcd}
\]
Then, using Lemma \ref{lem: smash in simp set} and Property \ref{property2}, we get either $d\cong 0$ or $d\cong *$ in both cases, which is a contradiction. Thus elements in $C(d)\Smash C(0)$ are not identified with elements in $C(0)\Smash C(d)$ in ${(C\Smash C)}(d)$.

Now we consider $C\Smash_R C$ instead of $C\Smash C$.
Some identifications do occur in the pointed space ${(C\Smash_R C)}(d)$ via the $R$-action structure maps. 
Recall that since $C$ is counital, we have the commutative diagram: 
\begin{equation}\label{diag: counital C smash C}
\begin{tikzcd}
C\Smash_R C \ar{r}{\mathsf{id}_C\Smash \varepsilon} & C\Smash_R R \cong C \cong R\Smash_R C & C\Smash_R C \ar{l}[swap]{\varepsilon \Smash \mathsf{id}_C}\\
& C.\ar[equals]{u}\ar[bend left]{ul}{\Delta} \ar[bend right]{ur}[swap]{\Delta} &
\end{tikzcd}
\end{equation}
The commutativity shows that the maps $\varepsilon \Smash \mathsf{id}_C$ and $\mathsf{id}_C\Smash \varepsilon$ are epimorphisms in $\T^{\D}$, hence surjective once evaluated at $d$. 
Suppose now that the structure map $\sigma: R(d)\wedge C(0)\rightarrow C(d)$ is not surjective. 
Consider an element $c$ in $C(d)$ not in the image of $\sigma$. 
On the one hand, we can view $C$ as $C \wedge_R R$ and so, if we consider $c$ in $C(d)$ as $c\Smash \lambda(1)$ in $C(d)\Smash R(0)$ of $(C\Smash_R R)(d)$, then we can take its preimage under the map:
\[\mathsf{id}_C\wedge \varepsilon: (C\wedge_R C)(d) \longrightarrow (C\wedge_R R)(d).\] 
The preimage lies in the contribution from $C(d)\wedge C(0)$ in (\ref{equ C wedge C in diag}).
On the other hand, if we view $C$ as $R\Smash_R C$, consider $c$ in $C(d)$ as $\lambda(1)\Smash c$ in $R(0)\Smash C(d)$ of $(R\Smash_R C)(d)$ and take its preimage under the map: \[\varepsilon \wedge \mathsf{id}_C: (C\wedge_R C)(d) \longrightarrow (R\wedge_R C)(d),\] it belongs to the copy $C(0)\wedge C(d)$ in (\ref{equ C wedge C in diag}). Since we have supposed that $\sigma$ is not surjective and that $c$ does not belong to the image of $\sigma$ in $C(d)$, we get that the two preimages ${(\mathsf{id}_C\wedge \varepsilon)}^{-1}(c)$ and ${(\varepsilon\wedge \id_C)}^{-1}(c)$ in  ${(C\wedge_R C)}(d)$ are disjoint from each other.
But then the commutativity of the diagram (\ref{diag: counital C smash C}) forces two different values of $\Delta(c)$ for the map $\Delta :C\longrightarrow C\wedge_R C$.
We get a contradiction, thus $\sigma : R(d)\wedge C(0)\rightarrow C(d)$ must be surjective.
\end{proof}

\begin{rem}
Notice that we did \emph{not} require $R(0)\cong S^0$ for Lemma \ref{lem: epimorphism}. It is valid for any commutative monoid $R$ in $\T^{\D}$. This suggests that, even though $R$-coalgebras in $\T^{\D}$ are not necessarily cocommutative when $R(0)\ncong S^0$, there are restrictions on the possibilities of $R$-coalgebras in $\T^{\D}$.
\end{rem}

\begin{rem}\label{rem: epi in symmetric spectra}
In the proof of Lemma \ref{lem: epimorphism} above, choosing a particular category $\D$ might simplify the reader's understanding. For instance, if we choose $\D=\ssigma_+$ as in \cite[Example 4.2]{MMSS}, where $\ssigma$ is the category of finite sets and their permutations, we obtain the usual category of symmetric spectra $\SpS:=\mathsf{Sp}^{\ssigma_+}_\S$ as in \cite{SS}. Lemma \ref{lem: epimorphism}, for $\S$-coalgebras in $\SpS$, appeared in an early version of~\cite{HS.cothh} before being developed further here.
The equation (\ref{equ: colimit coend}) of the internal smash product of $\ssigma_+$-spaces $X$ and $Y$ at the $n$th level simplifies to:
\[
(X\Smash Y)_n=\bigvee_{p+q=n} \Sigma_n^+ \wedge_{\Sigma_p\times \Sigma_q} X_p\wedge Y_q,
\]
where $\Sigma_n$ is the symmetric group on $n$ letters. Thus we see directly that the copies $X_n\Smash Y_0$ and $X_0\Smash Y_n$ are never identified in the internal smash product, but they can be when considering $(X\Smash_\S Y)_n$. In particular we do not need Lemma \ref{lem: smash in simp set} here.
\end{rem}

In order to prove Theorem \ref{thm: diag spectra coco}, we now investigate the condition \ref{cond1 of thm} of Theorem \ref{thm: prethm} with the functor $\Ev_0:\DT_R\rightarrow \T$.
We need the next two results.

\begin{lem}\label{lem: homeomorphism}
Let $X$ and $Y$ be pointed spaces.
For any commutative monoid $(R, \phi, \lambda)$ in $\T^{\D}$, and any object $d$ in $\D$, we have a homeomorphism of pointed spaces: \[R(d)\Smash X \Smash Y \cong \left[ (R\Smash X) \Smash_R (R\Smash Y)\right] (d).\]
\end{lem}

\begin{proof}
The proof follows by adjointness of the functor $R\Smash -:\T\rightarrow \DT_R$ as in \cite[Lemma 21.3]{MMSS}.
\end{proof}

\begin{lem}\label{lem: smash of sphere coalgebra}
Let $\D$ satisfy Properties {\normalfont\ref{property1}} and {\normalfont\ref{property2}}.
Let $(R,\phi, \lambda)$ be a commutative monoid in $\T^{\D}$, where $R(0)\cong S^0$. Let $(C, \Delta, \varepsilon)$ be an $R$-coalgebra in $\T^{\D}$.
Then there is a homeomorphism of pointed spaces: \[C(0)\Smash C(0) \cong {(C\Smash_R C)}(0).\]
\end{lem}

\begin{rem}
The result of Lemma \ref{lem: smash of sphere coalgebra} is automatic in most categories of interest, and does not require $C$ to be a coalgebra. Indeed, we have $(X\Smash Y)(0)\cong X(0)\Smash Y(0)$ for any  symmetric spectra or orthogonal spectra $X$ and $Y$. In other words, the functor $\Ev_0:\DT_R\rightarrow \T$ is strong symmetric monoidal when $R(0)\cong S^0$. 
But this is not true in general (for instance in $\mathcal{W}$-spaces {and $\Gamma$-spaces}, {see Examples \ref{ex: smash in W space} and {\ref{ex: smash in gamma}} below}). 
\end{rem}

\begin{proof}
Denote by $\sigma:R\Smash C(0)\rightarrow C$ the natural epimorphic map of Lemma \ref{lem: epimorphism}. Let us consider the map of $R$-modules $\sigma\Smash \sigma : \big((R\Smash C(0)\big) \Smash_R \big(R\Smash C(0)\big) \rightarrow C\Smash_R C$. If we evaluate on the unit object $0$ we get a map:
\[
\sigma\Smash \sigma:C(0)\Smash C(0)\cong \left[\big((R\Smash C(0)\big) \Smash_R \big(R\Smash C(0)\big)\right](0) \longrightarrow (C\Smash_R C)(0),
\]
where the left homeomorphism is induced by Lemma \ref{lem: homeomorphism} and $R(0)\cong S^0$. Recall that $(C\Smash_R C)(0)$ is obtained from $(C\Smash C)(0)$ by coequalizing the $R$-action, so that we get a surjective continuous pointed map $(C\Smash C)(0)\rightarrow (C\Smash_R C)(0)$.
The map $\sigma \Smash \sigma$ factors through the space $(C\Smash C)(0)$:
\[
\begin{tikzcd}
\left[\big((R\Smash C(0)\big) \Smash_R \big(R\Smash C(0)\big)\right](0) \ar{dr}[swap]{\sigma\Smash \sigma} \ar{r}{\sigma\Smash \sigma} & (C\Smash C)(0)\ar{d}\\
& (C\Smash_R C)(0).
\end{tikzcd}
\]
Recall that $(C\Smash C)(0)=\left.\left(\bigvee \D(e\otimes f, 0) \Smash C(e)\Smash C(f)\right)\right/\sim$. From Lemma \ref{lem: epimorphism}, given an element $\alpha\Smash c_1 \Smash c_2$ in $\D(e\otimes f, 0) \Smash C(e) \Smash C(f)$, there exists $r_1 \in R(e)$, $r_2\in R(f)$, $c_1', c_2'\in C(0)$ such that $\sigma (r_1\Smash c_1')=c_1$ and $\sigma(r_2\Smash c_2')=c_2$. Hence $(\sigma\Smash \sigma)(\alpha \Smash r_1 \Smash c_1' \Smash r_2 \Smash c_2')=\alpha\Smash c_1 \Smash c_2$. Thus the map $\sigma \Smash \sigma$ is surjective.

Next we show that the map $\sigma \Smash \sigma:C(0)\sm C(0)\rightarrow (C\sm_R C)(0)$ is injective. Explicitly, the map $\sigma\Smash \sigma$ sends $C(0)\Smash C(0)$ to the copy $\D(0\otimes 0, 0)\Smash C(0)\Smash C(0)$ in $(C\Smash_R C)(0)$ via the natural isomorphism $0\otimes 0\stackrel{\cong}\longrightarrow 0$. Since $C$ is counital, we can consider the maps of $R$-modules:
\[
\begin{tikzcd}
C\Smash_R C \ar{r}{\id_C\Smash \varepsilon} & C\Smash_R R \cong C, & C\Smash_R C\ar{r}{\varepsilon\Smash \id_C}& R\Smash_R C\cong C.
\end{tikzcd}
\] 
Evaluating again at $0$ and factoring the above maps through the product, we get a continuous pointed map:
\[
\begin{tikzcd}[column sep= 2cm]
(C\Smash_R C)(0) \ar{r}{(\id_C\Smash\varepsilon)\Smash (\varepsilon\Smash\id_C)} & C(0) \Smash C(0).
\end{tikzcd}
\]
Explicitly, the above map acts as the identity on the copy of $C(0)\Smash C(0)$ in $(C\Smash_R C)(0)$.  That is, the composite:
\[
\begin{tikzcd}[column sep= 2cm]
C(0)\Smash C(0)\ar{r}{\sigma\Smash \sigma} & (C\Smash_R C)(0) \ar{r}{(\id_C\Smash\varepsilon)\Smash (\varepsilon\Smash\id_C)} & C(0) \Smash C(0).
\end{tikzcd}
\]
is the identity by the counital property. Thus $\sigma\Smash \sigma$ is injective. Therefore $\sigma \Smash \sigma$ induces the desired homeomorphism with inverse $(\id_C\Smash\varepsilon)\Smash (\varepsilon\Smash\id_C)$.
\end{proof}

\begin{proof}[Proof of Theorem \ref{thm: diag spectra coco}]
We apply Theorem \ref{thm: prethm} to the adjoint pair of functors:
\[
\begin{tikzcd}
R\Smash -: \T \ar[shift left]{r} & \DT_R:\Ev_0.\ar[shift left]{l}
\end{tikzcd}\] 
Lemma \ref{lem: smash of sphere coalgebra} proves \ref{cond1 of thm}. Lemma \ref{lem: coalg in set} shows \ref{cond2 of thm}. Finally, Lemma \ref{lem: epimorphism} induces \ref{cond3 of thm}.
\end{proof}

We end this section with three examples.

\begin{ex} 
The functor $R\Smash-:\T\rightarrow \DT_R$ does not lift to an essentially surjective functor on comonoid objects. For instance, when $R$ is the sphere spectrum $\S$ there exist examples of $\S$-coalgebras that are not isomorphic to suspension spectra. 
For example, in symmetric spectra, given a space $Y$ and a quotient space $Y/B$, one can form a counital coalgebra $C$ with $C_0  = Y_+$, and $C_n = S^n \sm (Y/B)_+$ for $n>0$.  The counit map on level $n$, that is, $S^n \sm (Y/B)_+ \to S^n$, is induced by the map from $Y/B_+\rightarrow S^0$ that sends only the base point to the base point of $S^0$.  
\end{ex}

\begin{ex}\label{ex: smash in W space}
Let $(\W,\Smash, S^0)$ be the category of based spaces homeomorphic to finite CW complexes endowed with the usual smash product of spaces, as described in \cite[Example 4.6]{MMSS}. 
The sphere spectrum $\S: \W\rightarrow \T$ is defined as the strong symmetric monoidal faithful functor induced by inclusion. Recall from \cite[Lemma 4.9]{MMSS} that any $\W$-space has a unique structure of $\W$-spectrum over $\S$. Moreover, for any $\W$-spaces $X$ and $Y$, we have $X\Smash Y\cong X\Smash_\S Y$. We provide here an example where:
\[
X(S^0)\Smash Y(S^0)\ncong (X\Smash Y)(S^0),
\]
for some choice of $\W$-spaces $X$ and $Y$.
Recall from equation (\ref{equ: colimit coend}) that we have: 
\begin{equation}\label{eq: coequalizer in W}
(X\Smash Y)(S^0)=\left.\left(\bigvee_{(K,L)\in \W\times \W} \W(K\Smash L, S^0) \Smash X(K) \Smash Y(L)\right)\right/ \sim,
\end{equation}
where elements $\alpha\sm x_K\sm y_L$ in $\W(K\sm L, S^0)\sm X(K) \sm Y(Y)$ are identified with elements $x_0\sm y_0$ in $X(S^0)\Smash Y(S^0)\cong \W(S^0\sm S^0, S^0)\Smash X(S^0)\sm Y(S^0)$ if and only if either one of the following type of identifications occurs: 
\begin{enumerate}[label=\upshape\textbf{(\arabic*)}]
\item\label{enum 1} there exist maps $f:S^0\rightarrow K$ and $g:S^0\rightarrow L$ in $\W$ such that $X(f)(x_0)=x_K$, $Y(g)(y_0)=y_L$ and the following diagram commutes:
\[
\begin{tikzcd}
S^0\sm S^0 \ar{r}{f\sm g} \ar{dr}[swap]{\cong} & K\sm L \ar{d}{\alpha}\\
& S^0;
\end{tikzcd}
\]

\item\label{enum 2} there exist maps $f:K\rightarrow S^0$ and $g: L\rightarrow S^0$ in $\W$ such that $X(f)(x_K)=x_0$, $Y(g)(y_L)=y_0$ and the following diagram commutes:
\[
\begin{tikzcd}
K\sm L \ar{r}{f\sm g} \ar{dr}[swap]{\alpha} & S^0\sm S^0 \ar{d}{\cong}\\
& S^0;
\end{tikzcd}
\]
\item\label{enum 3} there exist spaces $K'$ and $L'$ in $\W$, together with a map $\beta: K'\sm L'\rightarrow S^0$ in $\W$, such that there exist maps $f:S^0\rightarrow K$, $g:S^0\rightarrow L$, $f':K\rightarrow K'$ and $g':L\rightarrow L'$ in $\W$ where $X(f)(x_0)=X(f')(x_K)$, $Y(g)(y_0)=Y(g')(y_L)$ and the following diagram commutes:
\[
\begin{tikzcd}
S^0\sm S^0 \ar{r}{f\sm g} \ar{dr}[swap]{\cong} & K'\sm L' \ar{d}{\beta} & K\sm L\ar{l}[swap]{f'\sm g'}\ar{dl}{\alpha}\\
& S^0. &
\end{tikzcd}
\]
\end{enumerate}

For any based space $A$, let us defined the infinite symmetric product of $A$, denoted $\SP(A)$, to be the free commutative monoid generated by $A$ in $(\T, \times, *)$, see more details in \cite[Definition 5.2.1]{aguilar}. 
Write $[a_1, a_2,\ldots, a_n]$ for the equivalence class in $\SP(A)$ of the point $(a_1, a_2, \ldots, a_n, *, *, \ldots)$ in $\prod_{n\geq 1} A$. 
This defines a functor $\SP:\T\rightarrow \T$ where a map $f:A\rightarrow B$  induces $\SP(f):\SP(A)\rightarrow \SP(B)$ defined by $f([a_1,\ldots, a_n]) =  [f(a_1), \ldots, f(a_n)]$.
We precompose by the functor $\S:\W\rightarrow \T$ to obtain a based, continuous functor $\SP:\W\rightarrow \T$, which makes the infinite symmetric product $\SP$ into a $\W$-space. 
A standard result of Dold-Thom gives that, for $n\geq 1$, the spaces $\SP(S^n)$ are the Eilenberg-Mac Lane spaces $K(\mathbb{Z}, n)$, see \cite[Proposition 6.1.2]{aguilar}, and thus $\SP$ is the Eilenberg-Mac Lane spectrum $H\mathbb{Z}$ in $\W$-spaces.

We argue here that $(\SP\sm \SP)(S^0)\ncong \SP(S^0)\sm \SP(S^0)$. Let us describe particular elements $\alpha\sm x_K\sm y_L$ in $\W(K\sm L, S^0)\sm \SP(K)\sm \SP(L)$, for particular $K$ and $L$ in $\W$, that are not identified to any element $x_0\sm y_0$ in $\SP(S^0)\sm \SP(S^0)$, via any identifications of the form \ref{enum 1}, \ref{enum 2} and \ref{enum 3} above. 
Let us first describe $K$ and $L$ and the map $\alpha:K\sm L\rightarrow S^0$. For the map $\alpha$ not be to be trivial, we choose $K$ and $L$ disconnected as follows. 
Let $K=\{*, k_1, k_2\}$ and $L=\{*, \ell_1, \ell_2\}$ be discrete spaces.
Then the smash product $K\sm L$ is given by the discrete space $\{ *, k_1\sm \ell_1, k_1\sm \ell_2, k_2\sm \ell_1, k_2\sm \ell_2\}$. We define the map $\alpha:K\sm L\rightarrow S^0$ by $\alpha(*)=0$ and:
\[
\begin{tikzcd}[column sep=small]
\alpha(k_1\sm \ell_1)=0, &\alpha(k_1\sm \ell_2)=1, &\alpha(k_2\sm \ell_1)=1,  & \alpha(k_2\sm \ell_2)=0.
\end{tikzcd}
\]

Given our choice of $\alpha:K\sm L\rightarrow S^0$, the reader can verify that there are no maps $f:K\rightarrow S^0$ and $g:L\rightarrow S^0$ in $\W$ such that they fit in the commutative diagram:
\begin{equation}\label{eq: example of W}
\begin{tikzcd}
K\sm L \ar{r}{f\sm g} \ar{dr}[swap]{\alpha}& S^0\sm S^0\ar{d}{\cong}\\
& S^0 .
\end{tikzcd}
\end{equation}
This shows that for any choice of $x_K$ and $y_L$ in $\SP(K)$ and $\SP(L)$, no element $\alpha\sm x_K\sm y_L$ is  identified with an element $x_0\sm y_0\in \SP(S^0)\sm \SP(S^0)$ in the identification of type \ref{enum 2} as described above. 

Let $x_K=[k_1, k_2]$. Let $y_L$ be any non-basepoint element of $\SP(L)$.
{Any element of $\SP(S^0)$ is of the form $[1,\ldots, 1]$ and, for a map $f:S^0\rightarrow K$, the map $\SP(f)$ sends this element to $[f(1), \ldots, f(1)]$ in $\SP(K)$.} 
Since $k_1\neq k_2$, there is no map $f:S^0\rightarrow K$ such that $\SP(f)(x_0)=x_K$ for some $x_0\in \SP(S^0)$. Thus, identifications of type \ref{enum 1} do not occur on the element $\alpha\sm x_K\sm y_L$ in $\W(K\sm L, S^0)\sm \SP(K)\sm \SP(L)$.

Suppose now there exist objects $K'$ and $L'$ in $\W$, together with maps $f':K\rightarrow K'$, $g':L\rightarrow L'$ and $\beta:K'\sm L'\rightarrow S^0$ in $\W$, such that the following diagram commutes:
\[
\begin{tikzcd}
K\sm L\ar{r}{f'\sm g'}\ar{dr}[swap]{\alpha} & K'\sm L'\ar{d}{\beta}\\
& S^0.
\end{tikzcd}
\]
Since $\alpha$ is non-trivial, it follows that $K'\sm L'$ is disconnected and $\beta$ is non-trivial. Since $K'\sm L'$ is disconnected, if follows that both $K'$ and $L'$ are disconnected as well.
We argue that the images $f'(k_1)$ and $f'(k_2)$ do not lie in the same path-component of $K'$. 
If we suppose they are, then commutativity of the above diagram gives:
\[
0=\alpha(k_1\sm \ell_1)=\beta(f'(k_1)\sm g'(\ell_1))=\beta(f'(k_2)\sm g'(\ell_1))=\alpha(k_2\sm \ell_1)=1,
\]
which is a contradiction. Also, since $\beta(f'(k_2)\sm g'(\ell_1))=\alpha(k_2\sm \ell_1)=1$, it follows that $f'(k_2)$ is not the basepoint of $K'$.
Since $\beta(f'(k_1)\sm g'(\ell_2))=\alpha(k_1\sm \ell_2)=1$, also $f'(k_1)$ is not the basepoint.
Thus the two non-basepoints $f'(k_1)\neq f'(k_2)$ in $K'$ are distinct. Whence, there are no elements $x_0\sm y_0\in \SP(S^0)\sm \SP(S^0)$ such that identifications of type \ref{enum 3} occur with the element $\alpha\sm x_K\sm y_L$.

Therefore, we have presented a non-trivial element $\alpha\sm x_K\sm y_L\in \W(K\sm L, S^0)\sm \SP(S^0)\sm \SP(S^0)$ in the summand $(\SP\sm \SP)(S^0)$ of the equation (\ref{eq: coequalizer in W}) that is not identified with any element in $\SP(S^0)\sm \SP(S^0)$. Thus $(\SP\sm \SP)(S^0)\ncong \SP(S^0)\sm \SP(S^0)$.
\end{ex}

\begin{rem}
In Example \ref{ex: smash in W space}, we purposefully have chosen a $\W$-space $X$ where the structure maps:
\[
K\sm X(S^0)\rightarrow X(K\sm S^0)\cong X(K),
\]
are not epimorphisms. However, if $X$ is an $\S$-coalgebra in $\W$-spaces, Lemma \ref{lem: epimorphism} shows that such non-epimorphic structure maps are not possible. 
\end{rem}

\begin{ex}\label{ex: smash in gamma}
Recall, from \cite[Example 4.8]{MMSS}, that if we choose $\D$ to be the category of finite based sets $n_+ = \{0,1,... ,n\}$ and all based maps, where 0 is the basepoint, then a $\D$-space is a $\Gamma$-space as in \cite{Segal}.
As in $\W$-spaces, here the sphere spectrum $\S: \D\rightarrow \T$ is defined as the strong symmetric monoidal faithful functor induced by inclusion (endowing the finite sets with the discrete topology). Recall from \cite[Lemma 4.9]{MMSS} that we have $X\sm_\S Y\cong X\sm Y$ for any $\Gamma$-spaces $X$ and $Y$, as the action of the sphere spectrum $\S$ provides no additional data.
As in Example \ref{ex: smash in W space}, the $\Gamma$-space model of the Eilenberg-Mac Lane spectrum $H\mathbb{Z}$ provides an example of a $\Gamma$-space $X$ where $(X\sm X)(1_+)\ncong X(1_+)\sm X(1_+)$. The proof is similar to Example \ref{ex: smash in W space}.
\end{ex}

\section{Coalgebras in EKMM-Spectra}\label{sec: ekmm}

We prove here (in Theorem \ref{thm: ekmm coco} below) a result similar to Theorem \ref{thm: diag spectra coco} for EKMM-spectra using the same strategy from Theorem \ref{thm: prethm}. We first investigate the definition of the smash product in this case.

Let us set notation and recall the definitions. All missing details can be found in \cite{EKMM}. 
Let $\mathcal{L}$ denote the category whose objects are universes and whose morphisms are linear isometries.
We fix $\mathcal{U}$ a universe, that is, a countable dimensional real inner product space.  We say $X$ is a spectrum indexed on $\mathcal{U}$ if we are given a collection $X_V$ of pointed topological spaces for each $V\subseteq \mathcal{U}$ a finite dimensional subspace, together with structure maps:
\[
\sigma_{V,W}:\Sigma^{W-V}X_V\longrightarrow X_W,
\]
whenever $V\subseteq W\subseteq \mathcal{U}$, such that the adjoint of each $\sigma_{V,W}$ is a homeomorphism. Given two spectra $X$ and $Y$ indexed on $\mathcal{U}$, their external smash product $X\overline{\wedge}Y$ is a spectrum indexed on $\mathcal{U}^{\oplus2}$ defined by:
\[
(X\overline{\wedge}Y)_{V}=X_{V_1}\wedge Y_{V_2},
\]
for any finite dimensional subspace $V=V_1\oplus V_2\subseteq \mathcal{U}^{\oplus2}$.
Let $\mathcal{L}(n)=\mathcal{L}(\mathcal{U}^{\oplus^n}, \mathcal{U})$ for all $n\in\mathbb{N}$. 
We say $X$ is an $\mathbb{L}$-spectrum if the spectrum $X$ is endowed with an action $\alpha:\mathcal{L}(1)\ltimes X \longrightarrow X$ (see \cite[Chapter I, Definition 4.2]{EKMM}). If $X$ and $Y$ are $\mathbb{L}$-spectra, their operadic smash product $X\Smash_{\mathcal{L}} Y$ is the coequalizer:
\begin{equation}\label{equ: operad act}
\begin{tikzcd}[column sep=large]
\Big(\mathcal{L}(2)\times \mathcal{L}(1) \times \mathcal{L}(1) \Big)\ltimes (X\overline{\wedge}Y)  \ar[shift left]{r}{\gamma \ltimes \mathsf{id}_{X\overline{\wedge}Y}}\ar[shift right]{r}[swap]{{\mathsf{id}_{\mathcal{L}(2)}\ltimes (\alpha_X\overline{\Smash}\alpha_Y)}} & \mathcal{L}(2) \ltimes X\overline{\wedge}Y,
\end{tikzcd}
\end{equation}
where $\gamma:\mathcal{L}(2)\times \mathcal{L}(1) \times \mathcal{L}(1) \rightarrow \mathcal{L}(2)$ is the map defined by: $(\theta, \varphi, \psi) \longmapsto \theta \circ (\varphi\oplus \psi)$,
and $\mathsf{id}_{\mathcal{L}(2)}\ltimes (\alpha_X\overline{\Smash}\alpha_Y)$ is defined via the isomorphism (see \cite[Chapter I, Proposition 2.2(ii)]{EKMM}):
\[
\Big(\mathcal{L}(2)\times \mathcal{L}(1)\times \mathcal{L}(1)\Big) \ltimes X\overline{\wedge}Y \cong \mathcal{L}(2) \ltimes\Big( (\mathcal{L}(1) \times \mathcal{L}(1)) \ltimes (X\overline{\wedge} Y)\Big),
\]
and $\alpha_X\overline{\Smash}\alpha_Y$ is the induced map: 
\[
\mathcal{L}(1)\times \mathcal{L}(1)\ltimes (X\overline{\wedge}Y) \cong (\mathcal{L}(1)\ltimes X) \overline{\wedge} (\mathcal{L}(1)\ltimes Y) \stackrel{\alpha_X\overline{\Smash}\alpha_Y}\longrightarrow X\overline{\wedge}Y.
\]
A definition of the twisted half-smash product $\mathcal{L}(2) \ltimes X\overline{\wedge}Y$ can be found in \cite[Definition 5.1]{Cole}. Let us make the construction explicit.
For each $V\subseteq \mathcal{U}^{\oplus2}$, define a Thom spectrum $\mathcal{M}(V)$ indexed on $\mathcal{U}$ so that, for any $V\subseteq W \subseteq \mathcal{U}^{\oplus2}$ we have an isomorphism of spectra indexed on $\mathcal{U}$: 
\[
\Sigma^{W-V} \mathcal{M}(W) \stackrel{\cong}\longrightarrow \mathcal{M}(V),
\]
see \cite[Proposition 4.3]{Cole}. Whenever the dimension of $U\subseteq \mathcal{U}$ is strictly smaller than the dimension of $V\subseteq \mathcal{U}^{\oplus2}$, the space $\mathcal{M}(V)_U$ is just a point (see beginning of \cite[Section 4]{Cole}).
The twisted half-smash product $\mathcal{L}(2)\ltimes (X\overline{\wedge} Y)$ is defined to be the colimit (see  \cite[Definition 3.5]{Cole}) in the category of spectra indexed over $\mathcal{U}$:
\begin{equation}\label{equ: half-smash product}
\mathsf{colim}_{V\subseteq \mathcal{U}^{\oplus2}}\Big(\mathcal{M}(V) \wedge {(X\overline{\wedge}Y)}_V\Big),
\end{equation}
where the colimit is taken over the maps: 
\begin{eqnarray*}
\mathcal{M}(V) \wedge {(X\overline{\wedge}Y)}_V & \cong & \Sigma^{W-V} \mathcal{M}(W) \wedge {(X\overline{\wedge}Y)}_V \\
& \cong & \mathcal{M}(W)\wedge \Sigma^{W-V}{(X\overline{\wedge}Y)}_V \\
& \rightarrow & \mathcal{M}(W) \wedge {(X\overline{\wedge} Y)}_W
\end{eqnarray*}
for $V\subseteq W \subseteq \mathcal{U}^{\oplus 2}$, finite dimensional subspaces. This amounts to saying that, for any finite dimensional subspace $U\subseteq \mathcal{U}$, the smash product  $(X\wedge_{\mathcal{L}}Y)_U$ can be regarded as a quotient of the following space: 
\begin{equation*}
\left.{\left( \bigvee_{\substack{V=V_1\oplus V_2\subseteq \mathcal{U}^{\oplus2}\\\text{finite dimensional}}} \mathcal{M}(V)_U \wedge X_{V_1}\wedge Y_{V_2}\right)}\right/{\sim}.
\end{equation*}
Since $\mathcal{M}(V)_U$ is a point whenever the dimension of $V$ is bigger than $U$ and that $\mathcal{M}(0)_0\cong S^0$, we obtain:
\begin{equation}\label{equ: smash of sphere in ekmm}
(X\Smash_\mathcal{L} Y)_0\cong X_0\Smash Y_0.
\end{equation}

Recall from \cite[Chapter III, Definition 1.1]{EKMM} that an $\mathbb{S}$-module $X$ is an $\mathbb{L}$-spectrum such that the natural map $\mathbb{S}\wedge_\mathcal{L}X \rightarrow X$ is an isomorphism. The smash product of $\S$-modules $X$ and $Y$ is defined as $X\Smash_\S Y=X\Smash_\mathcal{L}Y$ (see \cite[Chapter II, Definition 1.1]{EKMM}). 
We denote the resulting symmetric monoidal category by $\ekmm$.
Similarly as in previous sections, given a commutative monoid $R$ in $\ekmm$ (i.e. a commutative $\S$-algebra), we define the smash product $X \Smash_R Y$ of two $R$-modules $X$ and $Y$ as the coequalizer:
\begin{equation}\label{eqy: R-smash in ekmm}
\begin{tikzcd}[column sep=huge]
 X\Smash_\S R\Smash_\S Y \ar[shift left]{r}{\id_X\Smash \sigma}\ar[shift right]{r}[swap]{(\sigma\circ \tau) \Smash \id_Y} & X\Smash_\S Y.\end{tikzcd}
\end{equation}
See \cite[Chapter III, Definition 3.1]{EKMM}.

\begin{thm}\label{thm: ekmm coco}
Let $R$ be a commutative $\S$-algebra in $\ekmm$ such that $R_0\cong S^0$. Let $(C,\Delta,\varepsilon)$ be an $R$-coalgebra in $\ekmm$. Then $C$ is cocommutative.
\end{thm}

\begin{proof}
We apply Theorem \ref{thm: prethm} to the pair 
$\begin{tikzcd}
R\Smash -: \T \ar[shift left]{r} & \Mod_R(\ekmm):\Ev_0.\ar[shift left]{l}
\end{tikzcd}$ 
Equation (\ref{equ: smash of sphere in ekmm}) shows that $\Ev_0$ is a strong symmetric monoidal functor, and thus proves condition \ref{cond1 of thm}. Lemma \ref{lem: coalg in set}  shows \ref{cond2 of thm}. We conclude using Lemma \ref{lem: epi in ekmm} below and \ref{cond3 of thm} of Theorem \ref{thm: prethm}.
\end{proof}

\begin{lem}\label{lem: epi in ekmm}
Let $R$ be a commutative $\S$-algebra in $\ekmm$. Let $(C,\Delta, \varepsilon)$ be an $R$-coalgebra in $\ekmm$. Then the natural map $\sigma: R\Smash C_0\longrightarrow C$ is an epimorphism of $R$-modules in $\ekmm$.
\end{lem}

\begin{proof}
As we have recalled above, given any finite dimensional subspace $U\subseteq \mathcal{U}$, the smash product $(C\Smash_R C)_U$ can be regarded as: 
\begin{equation}\label{eq:ekmm}
(C\wedge_{R}C)_U=\left.{\left( \bigvee_{\substack{V=V_1\oplus V_2\subseteq \mathcal{U}^{\oplus2}\\ \text{finite dimensional}}} \mathcal{M}(V)_U \wedge C_{V_1}\wedge C_{V_2}\right)}\right/{{\sim_R}}.
\end{equation}
If $\dim(U)=1$ in equation (\ref{eq:ekmm}), we get that $(C\wedge_R C)_U$ is simply: 
\[
\left.{\left( \Big(\mathcal{M}(U\oplus 0)_U\wedge C_U \wedge C_0\Big) \vee \Big(\mathcal{M}(0)_U\wedge C_0\wedge C_0\Big) \vee\Big(\mathcal{M}(0\oplus U)_U\wedge C_0 \wedge C_U\Big)\right)}\right/{{\sim_R }.}
\]
Denote the action from the $\mathbb{L}$-structure of $C$ by $\alpha:\mathcal{L}(1)\ltimes C \longrightarrow C.$ Elements in $C_U \wedge C_0$ are not identified with elements in $C_0\wedge C_U$ in the coequalizer (\ref{equ: operad act}) via the the map $\alpha\overline{\Smash}\alpha$.
However, elements of $C_U\wedge C_0$ can be identified with elements of $C_0\wedge C_U$ under the $R$-action $R_U\Smash C_0\rightarrow C_U$ {(in the coequalizer (\ref{eqy: R-smash in ekmm}))} and the structure map $\sigma_{0,U}$ in $(C\wedge_R C)_U$ {(in the colimit (\ref{equ: half-smash product}))}. Notice the similarity with the equation (\ref{equ C wedge C in diag}) in previous section. We can argue similarly. Suppose the map $R_U\Smash C_0\rightarrow C_U$ is not surjective. Since the composite:
\[
\begin{tikzcd}
S^U\Smash C_0 \ar{r} & R_U\Smash C_0 \ar{r} & C_U,
\end{tikzcd}
\]
is the structure map $\sigma_{0,U}$, then $\sigma_{0,U}$ must also not be surjective.
Let $c$ be an element in $C_U$ not in the image of $R_U\Smash C_0\rightarrow C_U$. 
Notice that $c$ is also not in the image of $\sigma_{0,U}$.
Its preimage under the map $\varepsilon \wedge \mathsf{id}_C$ lies in $C_0\Smash C_U$ but its preimage under the map $\mathsf{id}_C  \Smash \varepsilon$ lies in $C_U\Smash C_0$. 
Since $c$ is not in the image of $\sigma_{0,U}$, elements in the two preimages ${(\mathsf{id}_C\wedge \varepsilon)}^{-1}(c)$ and ${(\varepsilon\wedge \id_C)}^{-1}(c)$ are not identified in the colimit (\ref{equ: half-smash product}). 
Similarly, since $c$ is not in the image of $R_U\Smash C_0\rightarrow C_U$, the elements in the two preimages are not identified in the coequalizer (\ref{eqy: R-smash in ekmm}).
Thus, the two preimages ${(\mathsf{id}_C\wedge \varepsilon)}^{-1}(c)$ and ${(\varepsilon\wedge \id_C)}^{-1}(c)$ are disjoint in $(C\Smash_R C)_U$. 
But we then get a contradiction with the commutativity of the diagram:
\[
\begin{tikzcd}
C\Smash_R C \ar{r}{\mathsf{id}_C\Smash \varepsilon} & C\Smash_R R \cong C \cong R\Smash_R C & C\Smash_R C \ar{l}[swap]{\varepsilon \Smash \mathsf{id}_C}\\
& C.\ar[equals]{u}\ar[bend left]{ul}{\Delta} \ar[bend right]{ur}[swap]{\Delta} &
\end{tikzcd}
\]
The higher dimensional cases are done similarly, so the map $R\Smash C_0\longrightarrow C$ is an epimorphism.
\end{proof}

\nocite{*}
\bibliographystyle{amsalpha}

\begin{thebibliography}{EKMM97}

\bibitem[AGP02]{aguilar}
Marcelo Aguilar, Samuel Gitler, and Carlos Prieto, \emph{Algebraic topology
  from a homotopical viewpoint}, Universitext, Springer-Verlag, New York, 2002,
  Translated from the Spanish by Stephen Bruce Sontz. \MR{1908260}

\bibitem[AM10]{AM}
Marcelo Aguiar and Swapneel Mahajan, \emph{Monoidal functors, species and
  {H}opf algebras}, CRM Monograph Series, vol.~29, American Mathematical
  Society, Providence, RI, 2010, With forewords by Kenneth Brown and Stephen
  Chase and Andr\'e Joyal. \MR{2724388}

\bibitem[And74]{Anderson}
D.~W. Anderson, \emph{Convergent functors and spectra}, Localization in group
  theory and homotopy theory, and related topics ({S}ympos., {B}attelle
  {S}eattle {R}es. {C}enter, {S}eattle, {W}ash., 1974), Lecture Notes in Math.,
  vol. 418, Springer, Berlin, 1974, pp.~1--5. \MR{0383388}

\bibitem[BF78]{BousfieldFrie}
A.~K. Bousfield and E.~M. Friedlander, \emph{Homotopy theory of {$\Gamma
  $}-spaces, spectra, and bisimplicial sets}, Geometric applications of
  homotopy theory ({P}roc. {C}onf., {E}vanston, {I}ll., 1977), {II}, Lecture
  Notes in Math., vol. 658, Springer, Berlin, 1978, pp.~80--130. \MR{513569}

\bibitem[Col97]{Cole}
Michael Cole, \emph{Appendix {A}: {T}wisted half-smash product and function
  spectra}, Rings, modules, and algebras in stable homotopy theory,
  Mathematical Surveys and Monographs, vol.~47, American Mathematical Society,
  1997.

\bibitem[EKMM97]{EKMM}
Anthony Elmendorf, Igor Kriz, Michael~A. Mandell, and Peter May, \emph{Rings,
  modules, and algebras in stable homotopy theory}, Mathematical Surveys and
  Monographs, vol.~47, American Mathematical Society, Providence, RI, 1997,
  With an appendix by M. Cole. \MR{1417719}

\bibitem[HS]{HS.cothh}
Kathryn Hess and Brooke Shipley, \emph{Topological co{H}ochschild homology
  (co{THH})}, in preparation.

\bibitem[HSS00]{SS}
Mark Hovey, Brooke Shipley, and Jeff Smith, \emph{Symmetric {S}pectra}, J.
  Amer. Math. Soc. \textbf{13} (2000), no.~1, 149--208. \MR{1695653}

\bibitem[Lur18]{lurie.elliptic}
Jacob Lurie, \emph{Elliptic cohomology {I}: Spectral abelian varieties},
  \url{http://www.math.harvard.edu/~lurie/papers/Elliptic-I.pdf}, February
  2018, preprint.

\bibitem[MM02]{MM}
Michael~A. Mandell and Peter May, \emph{Equivariant orthogonal spectra and
  {$S$}-modules}, Mem. Amer. Math. Soc. \textbf{159} (2002), no.~755, x+108.
  \MR{1922205}

\bibitem[MMSS01]{MMSS}
Michael~A. Mandell, Peter May, Stefan Schwede, and Brooke Shipley, \emph{Model
  categories of diagram spectra}, Proc. London Math. Soc. (3) \textbf{82}
  (2001), no.~2, 441--512. \MR{1806878}

\bibitem[Sch]{schwede-book}
Stefan Schwede, \emph{Symmetric spectra},
  \url{http://www.math.uni-bonn.de/people/schwede/SymSpec-v3.pdf}, electronic
  book.

\bibitem[Seg74]{Segal}
Graeme Segal, \emph{Categories and cohomology theories}, Topology \textbf{13}
  (1974), 293--312. \MR{0353298}

\bibitem[Shi04]{shipley-convenient}
Brooke Shipley, \emph{A convenient model category for commutative ring
  spectra}, Homotopy theory: relations with algebraic geometry, group
  cohomology, and algebraic {$K$}-theory, Contemp. Math., vol. 346, Amer. Math.
  Soc., Providence, RI, 2004, pp.~473--483. \MR{2066511}

\end{thebibliography}
\providecommand{\bysame}{\leavevmode\hbox to3em{\hrulefill}\thinspace}
\providecommand{\MR}{\relax\ifhmode\unskip\space\fi MR }
\providecommand{\MRhref}[2]{%
  \href{http://www.ams.org/mathscinet-getitem?mr=#1}{#2}
}
\providecommand{\href}[2]{#2}

\end{document}